\documentclass[leqno]{amsart}
\usepackage{amsmath,amsfonts,amsthm,amssymb,indentfirst}

\newcommand{\<}{\kern.0833em}

\newtheorem{theorem}{Theorem}
\newtheorem{lemma}[theorem]{Lemma}
\newtheorem{corollary}[theorem]{Corollary}
\newtheorem{proposition}[theorem]{Proposition}
\newtheorem{definition}[theorem]{Definition}

\renewcommand{\r}{\mathrm}
\newcommand{\hy}{\mbox{\!-}}
\newcommand{\Id}{\mathrm{Id}}
\newcommand{\id}{\mathrm{id}}
\newcommand{\cId}{\r{ch}\mbox{-}\Id}
\newcommand{\cid}{\r{ch}\mbox{-}\id}

\def\Dj{\mbox{\raise0.3ex\hbox{-}\kern-0.4em D}}

\makeatletter
\newcommand{\xlabel}{\stepcounter{equation}
  \gdef\@currentlabel{\p@equation\theequation}{\rm(\@currentlabel)}}
\makeatother
\newenvironment{xlist}
  {\begin{list}{\xlabel}
    {\setlength{\rightmargin}{20pt}
     \setlength{\leftmargin}{37pt}
     \setlength{\labelsep}{20pt}
     \setlength{\labelwidth}{20pt}}}
  {\end{list}}

\begin{document}

\begin{center}
\texttt{Comments, corrections, and related references welcomed, as
always!}\\[.5em]
{\TeX}ed \today
\vspace{2em}
\end{center}

\title{On lattices and their ideal lattices,
\mbox{and posets and their ideal posets}}
\thanks{The author is indebted to the referee for many helpful
corrections and references.\\
\hspace*{1.45em}After publication of this note, updates, further
references, errata, etc., if found, will be recorded at
http://math.berkeley.edu/\protect\linebreak[0]%
{$\!\sim$}gbergman\protect\linebreak[0]/papers/\,.
}
%		  	 poset semilat lat-ids
% Secondary:
%
\author{George M. Bergman}

\begin{abstract}
For $P$ a poset or lattice, let $\Id(P)$ denote the poset,
respectively, lattice, of upward directed downsets in $P,$ including
the empty set, and let $\id(P)=\Id(P)-\{\emptyset\}.$
This note obtains various results to the effect that $\Id(P)$ is always,
and $\id(P)$ often, ``essentially larger'' than $P.$
In the first vein, we find that a poset $P$ admits no
$\!<\!$-respecting map (and so in particular, no one-to-one
isotone map) from $\Id(P)$ into $P,$ and, going the other way, that
an upper semilattice $S$ admits no semilattice homomorphism from
any subsemilattice of itself onto $\Id(S).$

The slightly smaller object $\id(P)$ is known to be isomorphic to $P$
if and only if $P$ has ascending chain condition.
This result is strengthened to say that the only posets $P_0$ such that
for every natural number $n$ there exists a poset $P_n$ with
$\id^n(P_n)\cong P_0$ are those having ascending chain condition.
On the other hand, a wide class of cases is noted
where $\id(P)$ is embeddable in $P.$

Counterexamples are given to many variants of the statements proved.
\end{abstract}
\maketitle
% - - - - - - - - - - - - - - - - - - - - - - - - - - - - - -

\section{Definitions.}\label{S.def}

Recall that a poset $P$ is said to be {\em upward directed} if
every pair of elements of $P$ is majorized by some common element,
and that a {\em downset} in $P$ means a subset $d$ such
that $x\leq y\in d\implies x\in d.$
The downset generated by a subset $X\subseteq P$ will be
written $P\downarrow X=\{y\in P\mid (\exists\,x\in X)\ y\leq x\}.$
A {\em principal} downset means a set of
the form $P\downarrow\{x\}$ for some $x\in P.$

\begin{definition}\label{D.Id}
If $P$ is a poset, an {\em ideal} of $P$ will mean a
\textup{(}possibly empty\textup{)} upward directed downset in $P.$
The set of all ideals of $P,$ partially ordered by inclusion,
will be denoted $\Id(P),$ while we shall write
$\id(P)$ for the subset of {\em nonempty} ideals.

The subposet of $\Id(P),$ respectively $\id(P),$ consisting
of ideals generated by chains, respectively, nonempty chains, will be
denoted $\cId(P),$ respectively $\cid(P).$
\end{definition}

This use of the term ``ideal'' is common in lattice
theory, where an ideal of a lattice $L$ can be characterized
as a subset (often required to be nonempty) that is closed under
internal joins, and under meets with arbitrary elements of $L.$
For general posets, ``ideal'' is used in some works, such as
\cite{MEr}, with the same meaning as here;
in others, such as \cite{M+M+T}, ``(order) ideal'' simply means downset,
while in still others, such as \cite{OF} and \cite{MEr1},
an ``(order) ideal'' means a {\em Frink ideal}, which can
be characterized as an upward directed union of {\em intersections}
of principal downsets.
(We shall not consider Frink ideals here.
In upper semilattices, they are the same as our ideals.
For a general study of classes of downsets in posets, see \cite{MEr1}.)

If $S$ is an upper semilattice, its ideals are the closed sets
with respect to a closure operator, so $\Id(S)$ is a complete lattice.

If $L$ is a lattice (or a downward directed upper
semilattice), $\id(L)$ is a sublattice
of $\Id(L),$ though not a complete one unless
$L$ has a least element.
For $S$ any upper semilattice, $\id(S)$ at least
forms an upper subsemilattice of the lattice $\Id(S).$

In a poset $P,$ the principal downsets (which we can now also
call the principal ideals) form a poset isomorphic to $P.$
If $P$ has ascending chain condition, we see that every nonempty
ideal is principal, so $\id(P)\cong P.$
(This yields easy examples where
$\Id(P)$ is neither an upper nor a lower semilattice.)

The operators $\cId$ and $\cid$ are not as nicely behaved
as $\Id$ and $\id.$
Even for $L$ a lattice,
$\cId(L)$ need be neither an upper nor a lower semilattice.
For instance, regarding $\omega$ and $\omega_1$
(the first infinite and the first uncountable ordinals),
with their standard total orderings,
as lattices, let $L$ be the direct product lattice
$(\omega+1)\times(\omega_1+1).$
(Recall that each ordinal is taken to be the set of
all lower ordinals.
Thus, $\omega+1=\omega\cup\{\omega\}$
and $\omega_1+1=\omega_1\cup\{\omega_1\}.)$
Then the chains $x_0=\omega\times\{0\}$ and $x_1=\{0\}\times\omega_1$
belong to $\cId(L),$ but their join in $\Id(L),$ namely
$\omega\times\omega_1,$ has no cofinal subchain
(because $\omega$ and $\omega_1$ have different cofinalities),
hence does not lie in $\cId(L).$
Indeed, $x_0$ and $x_1$ have no least upper bound in $\cId(L),$
since their two common upper bounds
$y_0=L\downarrow(\omega\times\{\omega_1\})$ and
$y_1=L\downarrow(\{\omega\}\times\omega_1)$ intersect
in the non-chain-generated ideal $\omega\times\omega_1.$
One also sees from this that the latter two
elements, $y_0$ and $y_1,$ have no greatest lower bound in $\cId(L).$

Why are we considering these badly behaved operators?
Because the method of proof of our first result involves,
not merely an ideal, but an ideal generated by a chain, and it seemed
worthwhile to formulate the result so as to capture the
consequences of this fact.
An appendix, \S\ref{S.chains}, notes variants of these
constructions that are better behaved.

We will also use

\begin{definition}\label{D.str_iso}
A map $f:P\to Q$ of posets will be called {\em strictly isotone}
if $x<y$ in $P$ implies $f(x)<f(y)$ in $Q.$
\end{definition}

Thus, the strictly isotone maps include the embeddings of posets,
and so in particular, the embeddings of lattices and
of upper semilattices.

\section{Nonembeddability results.}\label{S.nonemb}

S.\,Todorcevic has pointed out to me that my first result,
Theorem~\ref{T.xembcIdP} below,
is a weakened version of an old result of {\Dj}.\,Kurepa \cite{DK},
which says that the poset of well-ordered chains in any poset
$P,$ ordered by the relation of one chain being an initial
segment of another, cannot be mapped into $P$ by a strictly isotone map.
(A still stronger version appears in Todorcevic \cite{ST}.)
One could say that the one contribution of Theorem~\ref{T.xembcIdP},
relative to Kurepa's result, is that by weakening the assertion about
chains to one about the downsets they generate,
it gives us a statement about ideals of $P.$

All the versions of this result have essentially the same proof;
it is short and neat, so we include it.
We give the result for posets; that statement implies, of course,
the corresponding statements for semilattices and lattices.

\begin{theorem}[cf.\ {\cite[Th\'eor\`eme~1]{DK}},
{\cite[Theorem~5]{ST}}]\label{T.xembcIdP}
Let $P$ be any poset.
Then there exists no strictly isotone map $f:\cId(P)\to P.$
Hence \textup{(}weakening this statement in two ways\textup{)}
there exists no embedding of posets $f:\Id(P)\to P.$
\end{theorem}

\begin{proof}
Suppose $f:\cId(P)\to P$ is strictly isotone.
Let us construct a chain of elements $x_\alpha\in P,$ where
$\alpha$ ranges over all ordinals of cardinality $\leq\r{card}(P)$
(i.e., over the successor cardinal to $\r{card}(P)),$
by the single recursive rule
\begin{xlist}\item\label{x.x_*a}
$x_\alpha\ =\ f(P\downarrow\{x_\beta\mid\beta<\alpha\}).$
\end{xlist}
Given $\alpha,$ and assuming recursively that
$\beta\mapsto x_\beta$ $(\beta<\alpha)$ is a strictly
isotone map $\alpha\to P,$ we see that
$P\downarrow\{x_\beta\mid\beta<\alpha\}$ is a member of $\cId(P),$
so~(\ref{x.x_*a}) makes sense.
We also see that for all $\beta<\alpha,$ the chain in $P$ occurring
in the definition of $x_\beta$ is a proper initial segment of the
chain in the definition of $x_\alpha;$ so the strict isotonicity
of $f$ insures that $x_\beta<x_\alpha,$
and our recursive assumption carries over to $\alpha+1.$
It is also clear that if that assumption holds for all $\beta$ less
than a limit ordinal $\alpha,$ it holds for $\alpha$ as well.

This construction thus yields a chain of cardinality
$>\r{card}(P)$ in $P,$ a contradiction, completing the proof.
\end{proof}

In the above proof, we restricted~(\ref{x.x_*a}) to ordinals
$\alpha$ of cardinality $\leq\r{card}(P)$ only so as to have a
genuine set over which to do recursion.
The reader comfortable with recursion on the proper class of
all ordinals can drop that restriction, ending the proof
with an all the more egregious contradiction.

Theorem~\ref{T.xembcIdP} is reminiscent of Cantor's result that the
power set of a set $X$ always has larger cardinality than $X.$
(Cf.\ the title of \cite{G+D}, where a similar result is proved with
the poset of all downsets in place of the smaller poset of ideals.)
In some cases, for instance when $P$ is the chain of rational numbers,
$\Id(P)$ in fact has larger cardinality than $P;$ but in others, for
instance when $P$ is the chain of integers, or of reals, it does not.
For the latter case, one can verify by induction that for every
natural number $n,$ the result of iterating this construction $n$
times, $\Id^n(\mathbb{R}),$ may be described
as the chain gotten by taking $\mathbb{R}\times(n+1),$ lexicographically
ordered, and attaching an extra copy of the chain $n$ to each end.
So the above theorem yields the curious fact
that the chain so obtained using a larger value of
$n$ can never be embedded in the chain obtained using a smaller value.
(The copies of $n$ at the top and bottom are irrelevant to this
fact, since by embedding $\mathbb{R}$ in, say, the interval
$(0,1),$ one can get an identification of $\Id^n(\mathbb{R})$ with a
``small'' piece of itself, hence in
particular, an embedding into itself minus those add-ons.)

Since the proof of Theorem~\ref{T.xembcIdP} is based on constructing
chains, one may ask whether $\cId(P)$ always contains a chain
that cannot be embedded in $P.$
That is not so; to see this, let us form a disjoint union of chains of
finite lengths $1,\,2,\,3,\ \dots,$ with no order-relations between
elements of different chains, and -- to make our example not only a
poset but a lattice -- throw in a top element and a bottom element.
The resulting lattice $L$ has ascending chain condition, hence
$\Id(L),$ and so also $\cId(L),$ consists of the principal ideals and
the empty ideal; in other words $\Id(L)=\cId(L)$ is, up to isomorphism,
the lattice obtained by attaching one new element to the bottom of $L.$
Hence, like $L,$ it has chains of all natural number lengths
and no more, though as
Theorem~\ref{T.xembcIdP} shows (and a little experimenting confirms),
it cannot be mapped into $L$ by any strictly isotone map.

In contrast to what Theorem~\ref{T.xembcIdP} says about $\Id(P),$
we noted in \S\ref{S.def} that $\id(P)$ is canonically
isomorphic to $P$ whenever the latter has ascending chain condition.
D.\,Higgs~\cite{DH}, answering a question of G.\,Gr\"atzer, showed
for lattices $L$ that it is {\em only} in this case that $\id(L)$
can be isomorphic in {\em any} way to $L,$ and
M.\,Ern\'e~\cite{MEr} (inter alia) generalized this
statement to arbitrary posets.
But our next result, extending the trick
of the preceding paragraph, shows that the class of lattices $L$
such that $\id(L)$ can be {\em embedded} as a lattice in $L$
(and hence the class of posets $P$ such that $\id(P)$
can be embedded as a poset in $P)$ is much larger.

\begin{proposition}\label{P.id_emb}
Every lattice $L$ is embeddable as a lattice in a lattice $L'$ such
that $\id(L')$ is embeddable as a lattice in $L'.$

Hence the same is true with ``lattice'' everywhere replaced
by ``upper semilattice'' or ``poset''.
\end{proposition}

\begin{proof}
Without loss of generality, assume $L$ nonempty.
Let $L'$ be the poset obtained by taking the disjoint union of the
lattices $L,\ \id(L),\ \dots,\ \id^n(L),\ \dots$ $(n\in\omega),$
with no order-relations among elements of distinct pieces,
and then throwing in a top element $1$ and a bottom element $0.$
It is easy to see that $L'$ is a lattice, and that every nonempty
ideal of $L',$ other than $L'$ and $\{0\},$ contains elements of
the sublattice $\id^n(L)$ for one and only one value of $n.$
For each $n,$ the ideals of this sort containing elements
of $\id^n(L)$ form a sublattice of $\id(L')$ isomorphic to
$\id(\id^n(L))=\id^{n+1}(L).$
One sees from this that $\id(L')$ is isomorphic to the sublattice
of $L'$ obtained by deleting the original copy of $L.$

This proves the assertion about lattices.
The corresponding statements for upper semilattices and for posets
follow, since every semilattice or poset can be embedded
as a subsemilattice or subposet in a lattice; e.g., in its lattice of
ideals in the former case, in its lattice of downsets in the latter.
(In fact, there exist embeddings preserving all least upper
bounds and greatest lower bounds that exist in the given structures:
\cite{HM}, \cite[Theorem V.21]{GBk}.)
\end{proof}

On the other hand, there are many posets $P$ for which we can deduce
from Theorem~\ref{T.xembcIdP} the non\-embeddability of $\id(P)$ in $P.$

\begin{corollary}[to Theorem~\ref{T.xembcIdP}]\label{C.xembcidP}
Suppose $P$ is a poset which admits a strictly isotone map $g$ into a
nonmaximal principal up-set within
itself, i.e., into $P\uparrow x$ for some nonminimal $x\in P.$
Then there exists no strictly isotone map $f:\id(P)\to P.$
In particular for $P$ the lattice of all subsets of an infinite set, or
of all equivalence relations on an infinite set, there is no such $f.$
\end{corollary}

\begin{proof}
By assumption we have a strictly isotone map $g:P\to P\uparrow x,$
where $x$ is not minimal.
Take $y<x$ in $P.$
If there existed a strictly isotone map $f:\id(P)\to P,$ then $gf$
would be another such map, with image consisting of elements $>y.$
Hence we could extend it to $\Id(P)$ by sending $\emptyset$
to $y,$ contradicting Theorem~\ref{T.xembcIdP}.
This proves our main assertion.

If $X$ is an infinite set, take distinct elements $x_0, x_1\in X.$
Then the lattice of all subsets of $X$ is isomorphic to its sublattice
consisting of subsets that contain $x_0,$ and the lattice of
equivalence relations is isomorphic to its sublattice
of equivalence relations that identify $x_0$ with $x_1.$
Thus, both lattices satisfy the hypothesis
of our main assertion, giving the final statement.
\end{proof}

F.\,Wehrung \cite{FW} shows that the lattice $L$ of equivalence
relations on a set of infinite cardinality $\kappa$ contains a
coproduct of two copies of itself (and hence, by results of
\cite{embed}, a coproduct of $2^\kappa$ copies of itself).
His proof uses the description of $L,$ up to isomorphism,
as $\id(L_\r{fin}),$ where $L_\r{fin}\subseteq L$ is the sublattice
of finitely generated equivalence relations.
This led me to wonder whether $L$ might also contain a copy
of $\id(L),$ and so initiated the present investigation.
The above corollary answers that question in the negative.

\section{Nonexistence of surjections.}\label{S.nosurj}

Another version of the idea that a lattice $L$ is essentially
smaller than its ideal lattice would be to say that there are no
surjective homomorphisms $L\to \Id(L).$
The next theorem shows that this is true.
We again get the result for a wider class of objects than
lattices, in this case upper semilattices.
We shall see that the result does not extend
to general posets or isotone maps, nor can we replace
ideals by chain-generated ideals; in these ways it is of a weaker
sort than Theorem~\ref{T.xembcIdP}.
On the other hand, it is stronger in a different~way.

\begin{theorem}\label{T.xsurjId}
Let $S$ be an upper semilattice.
Then there exists no upper semilattice homomorphism from
any {\em subsemilattice} $S_0\subseteq S$ onto $\Id(S).$
\end{theorem}

\begin{proof}
Suppose $f: S_0\to\Id(S)$ were such a surjective homomorphism.
Then we could map $\Id(\Id(S))$ to $\Id(S)$ by taking
each $I\in \Id(\Id(S))$ to $S\downarrow f^{-1}(I).$
Because $f$ is onto, distinct ideals $I$ of $\Id(\Id(S))$ yield
distinct ideals $f^{-1}(I)$ of $S_0,$ and these will
generate distinct ideals of $S.$
This leads to an embedding $\Id(\Id(S))\to\Id(S)$ as posets,
contradicting Theorem~\ref{T.xembcIdP}.
\end{proof}

We cannot replace the semilattice $S$ and semilattice homomorphism
$f$ in Theorem~\ref{T.xsurjId}
by a poset and an isotone map, because the
inverse image of an ideal under an isotone map $f$ need not be an ideal.
Indeed, we can get a counterexample to the resulting
statement in which the given poset is a lattice
$L,$ and $f$ is a strictly isotone bijection $L\to\Id(L)\!:$
Let $L$ consist of a greatest element $1,$ a least element $0,$ and
countably many mutually incomparable elements $a_n$ $(n\in\omega)$
between them, and let $f$ act by
\begin{xlist}\item[]
$f(1)=L,\quad f(a_{n+1})=\{a_n,0\}\ \ (n\in\omega),\quad
f(a_0)=\{0\},\quad f(0)=\emptyset.$
\end{xlist}
(If we try to apply the construction in the proof of
Theorem~\ref{T.xembcIdP} to the map
$S\downarrow f^{-1}(-)$ from $\Id(\Id(S))$ to subsets of $S,$
the values we get for $x_0,\ x_1,\ x_2,\ x_3$ are respectively
$\emptyset,\ \{0\}, \ \{a_0,0\},$ and $\{a_1,a_0,0\},$ of which the
last is not an ideal, so the construction cannot be continued further.)

We could, of course, get a version of Theorem~\ref{T.xsurjId}
for posets by restricting our morphisms to isotone maps
under which inverse images of ideals are ideals.

Alternatively, we can escape these difficulties if we are
willing to replace ideals by downsets, getting the first sentence
of the next result.
But in fact, we can deduce using Theorem~\ref{T.xsurjId}
a stronger statement, the second sentence.

\begin{corollary}\label{C.xsurjdown}
No isotone map from a subset $P_0$ of a poset $P$ to the
lattice $\r{Down}(P)$ of all downsets of $P$ is surjective.

In fact, no isotone map $f$ from a poset $P_0$ to any upper
semilattice $T$ containing $\r{Down}(P_0)$ as a subsemilattice
has the property that $f(P_0)$ generates $T$ as an upper semilattice.
\end{corollary}

\noindent
{\it Sketch of proof.}
Clearly the first assertion is a case of the second.
To prove the latter, let us,
for any poset $P,$ write $\r{fdown}(P)$ for the upper semilattice
of finite nonempty unions of principal downsets of $P.$
Then one can verify that
\begin{xlist}\item\label{x.fd}
$\r{fdown}(P)\ \cong$ the upper semilattice freely generated by the
poset $P.$
\end{xlist}
\begin{xlist}\item\label{x.Idfd}
$\r{Down}(P)\ \cong\ \Id(\r{fdown}(P)).$
\end{xlist}

Now given a poset $P_0$ and
an upper semilattice $T$ containing $\r{Down}(P_0),$ we see
from~(\ref{x.fd}) with $P_0$ for $P$ that isotone maps
$f:P_0\to T$ such
that $f(P_0)$ generates the semilattice $T$ are equivalent to surjective
semilattice homomorphisms $f':\r{fdown}(P_0)\to T.$
Hence, given such a map $f,$ if $T$ contains $\r{Down}(P_0)$ as a
subsemilattice, then the
inverse image under $f'$ of that subsemilattice
will be a subsemilattice of $\r{fdown}(P_0)$ which $f'$ maps
surjectively to $\r{Down}(P_0)\cong\Id(\r{fdown}(P_0)).$
But this is impossible, by Theorem~\ref{T.xsurjId}
with $S=\r{fdown}(P_0).$\qed\vspace{.5em}

We mentioned that one cannot replace $\Id(S)$ by $\cId(S)$
in Theorem~\ref{T.xsurjId}.
Indeed, even if we bypass the problem that $\cId(S)$ is not in general
an upper semilattice by restricting ourselves to cases where it is, the
proof of that theorem fails because $f^{-1}$ of an ideal
generated by a chain need not be generated by a chain.
Here is a counterexample to that generalization of the theorem.

\begin{lemma}\label{L.*k^ha0}
Let $\kappa$ be an infinite cardinal, and $S$ the lattice of all
finite subsets of $\kappa.$
Then $\cId(S)$ forms a lattice, and if $\kappa=\lambda^{\aleph_0}$
for some cardinal $\lambda,$ then $\cId(S)$ is a homomorphic image
of $S$ as an upper semilattice.
\end{lemma}

\noindent
{\it Sketch of proof.}
Note that as an upper semilattice with $0,$ $S$ is free on $\kappa$
generators, and that it has no uncountable chains.
 From the latter fact one can verify that $\cid(S)$ is isomorphic to the
poset of all countable subsets of $\kappa,$ which is again a lattice,
and has cardinality $\kappa^{\aleph_0}.$
Hence $\cId(S)$ is also a lattice of that cardinality.
If $\kappa=\lambda^{\aleph_0},$
then $\kappa^{\aleph_0}=\kappa,$ so as an upper
semilattice, $\cId(S)$ is a homomorphic image of the free
upper-semilattice-with-$\!0\!$ on $\kappa$ generators,
namely~$S.$\qed\vspace{.5em}

But I do not know whether, if $L$ is a lattice such that $\cId(L)$ is
again a lattice, the latter can ever be a {\em lattice-theoretic}
homomorphic image of $L,$ or of a sublattice thereof.
\vspace{.5em}

As another way of tweaking our results, we might go back to
Theorem~\ref{T.xembcIdP}, and try replacing $P$ on the right side
of the map $f$ by an isotone
or (if $P$ is a lattice or upper semilattice)
a lattice- or semilattice-theoretic homomorphic image of $P$
--~the dual of our use of a subsemilattice $S_0$ on the left-hand
side of the map in Theorem~\ref{T.xsurjId}.
It turns out that the sort of statements one can express in this
way are weakened versions of statements of the sort exemplified by
Theorem~\ref{T.xsurjId}.
For to embed an algebraic structure $A$ in a homomorphic image
of a structure $B$ is equivalent to giving an isomorphism between
$A$ and a subalgebra of that homomorphic image of $B;$ and
the subalgebras of homomorphic images of $B$ are a subclass of
the homomorphic images of subalgebras of $B,$ so we end up looking
at homomorphisms from subalgebras of $B$ onto $A,$ as
in Theorem~\ref{T.xsurjId}.

So, for instance, it follows from Theorem~\ref{T.xsurjId} that
if we restrict Theorem~\ref{T.xembcIdP} to semilattices $S$ and
semilattice homomorphisms, and replace $\cId$ with $\Id,$
then we {\em can} replace the codomain $S$ of our map by an
arbitrary semilattice homomorphic image of $S.$
In the opposite direction, Lemma~\ref{L.*k^ha0}
shows that if we {\em keep} the operator $\cId$ in
Theorem~\ref{T.xembcIdP}, and again assume $P$
and $\cId(P)$ to be semilattices and restrict $f$
to be a semilattice homomorphism, we cannot replace the codomain
by such an image of itself.
(In this case, the distinction between ``subalgebra of a homomorphic
image'' and ``homomorphic image of a subalgebra'' makes no difference,
for two reasons: semilattices have the Congruence
Extension Property, and in that example, the subalgebra
was the whole semilattice anyway.
So our counterexample to the statement modeled on
Theorem~\ref{T.xsurjId} is indeed a counterexample to what would
otherwise be the weaker statement modeled on Theorem~\ref{T.xembcIdP}.)

For posets, one has many possible variants of our results,
because of the many sorts of maps one can define among them.
E.g., we found it natural to prove Theorem~\ref{T.xembcIdP}
for strictly isotone (but not necessarily one-to-one) maps;
while the authors of \cite{G+D} show that no
poset $P$ admits a {\em one-to-one} map $\r{Down}(P)\to P$ that is
{\em either} $\!\leq\!$-preserving (i.e., isotone), {\em or}
$\!\not\leq\!$-preserving.
By Lemma~\ref{L.*k^ha0}, one cannot, in Theorem~\ref{T.xembcIdP},
replace the codomain poset $P$ by a general isotone image of itself;
but such a result might be true for images of other sorts.

\section{$P_0\cong\dots\cong\,\id^n(P_n)\cong\dots$ can only
happen in ``the obvious way''.}\label{S.id^oo}

We have mentioned that by Ern\'e's generalization~\cite{MEr}
of a result of Higgs~\cite{DH}, the only posets $P$ admitting any
isomorphism with $\id(P)$ are those for which the canonical embedding
$P\to\id(P)$ is an isomorphism, namely the posets
with ascending chain condition.
We prove below a further generalization of this statement.
Rather than assuming an isomorphism between $P$ and its {\em own} ideal
poset, we shall see that it suffices to assume $P$ simultaneously
isomorphic to an ideal-poset $\id(P_1),$ a double ideal-poset
$\id^2(P_2),$ and generally to an $\!n\!$-fold ideal-poset
$\id^n(P_n)$ for each $n.$
I will give two proofs: one based on the ideas of Higgs' and Ern\'e's
proofs, and one that obtains the result from Ern\'e's (via a version
of the trick of Proposition~\ref{P.id_emb} above).

First, some terminology and notation.
Generalizing slightly the language of
\cite{MEr2}, let us call an element $x$ of a poset $P$
{\em compact} if for every directed subset $S\subseteq P$ which has a
least upper bound $\bigvee S$ in $P,$ and such that $\bigvee S\geq x,$
there is some $y\in S$ which already majorizes $x.$
For $P$ any poset, the compact elements of
$\id(P)$ are the principal ideals.
(These are clearly compact, while a nonprincipal ideal is the join of
the directed system $S$ of its principal, hence proper, subideals.)
Thus, defining $d_P:P\to\id(P)$ by
\begin{xlist}\item\label{x.d_P}
$d_P(x)\ =\ P\downarrow\{x\},$
\end{xlist}
the map $d_P$ is an isomorphism between $P$ and the
poset of compact elements of $\id(P).$
Since the set of compact elements of a poset is determined
by the order structure of the poset, this
shows that $P$ and the map $d_P:P\to\id(P),$ are
recoverable, up to isomorphism, from the order structure of $\id(P).$

\begin{lemma}\label{L.n-cmpct}
Let us call the compact elements of a poset $P$
the {\em $\!1\!$-compact} elements,
and inductively define the {\em $\!n\!$-compact} elements of $P$ to be
the elements of the subposet of $\!n{-}1\!$-compact elements that
are compact in that subposet.
Then in a poset of the form $\id^n(P)$ where $n>1,$ every
non-compact element $a_0$ yields a chain
\begin{xlist}\item\label{x.a_0-n-1}
$a_0\ <\ a_1\ <\ \dots\ <\ a_{n-1},$
\end{xlist}
where for $i=1,\dots,n-1,$ $a_i$ is the least $\!i\!$-compact element
of $\id^n(P)$ majorizing $a_{i-1}.$
\end{lemma}

\begin{proof}
 From our preceding observations, we see that the $\!1\!$-compact
elements of $\id^n(P)$ are, in the notation of~(\ref{x.d_P}),
the members of
$d_{\id^{n-1}(P)}(\id^{n-1}(P)),$ the $\!2\!$-compact elements are the
members of $d_{\id^{n-1}(P)}d_{\id^{n-2}(P)}(\id^{n-2}(P)),$ and so on,
through the $\!n\!$-compact elements, which comprise the
set $d_{\id^{n-1}(P)}\dots d_{\id(P)}d_P(P).$

Note also that for any poset $Q,$ if $I$ is a nonprincipal ideal
of $\id(Q),$ then $d_Q^{-1}(I)$ i.e., $\{x\in Q\mid d_Q(x)\in I\},$
must be a nonprincipal ideal of $Q$ (though the converse is not true).
Moreover, that ideal, regarded as a member of $\id(Q),$ will majorize
all members of $I,$ and will be the least element that does so;
hence in $\id^2(Q),$ the element $d_{\id(Q)}d_Q^{-1}(I)$ will be
the least compact element majorizing the noncompact
element $I.$
Thus, in $\id^2(Q),$ every {\em noncompact} element has a least
compact element majorizing it, and that {\em compact} element is again
{\em noncompact} within the subposet of compact elements.

Hence in the situation of the lemma, where $a_0$ is a {\em noncompact}
element of an $\!n\!$-fold ideal poset $\id^n(P),$ we have a
least compact element $a_1$ majorizing it, which is the image
under $d_{\id^{n-1}(P)}$ of a {\em noncompact} element
of $\id^{n-1}(P),$ for which we can repeat the argument if $n-1>1,$
giving the desired chain~(\ref{x.a_0-n-1}).
\end{proof}

\begin{theorem}[{cf.\ \cite{DH}, \cite{MEr}}]\label{T.id^n}
Suppose $P_0$ is a poset such that for each natural number $n$ there
exists a poset $P_n$ with $P_0\,\cong\,\id^n(P_n).$
Then $P_0$ has ascending chain condition.
\end{theorem}

\noindent
{\it Proof 1.}
For notational simplicity, let us assume
without loss of generality that $P_0=\id(P_1).$
If $P_0$ does not have ascending chain condition, then
the poset $P_1$ clearly cannot have ascending chain
condition either; hence it has a nonprincipal ideal,
hence by Zorn's Lemma we can find a maximal nonprincipal
ideal, so $P_0$ will have a maximal noncompact element $a_0.$
Applying the preceding lemma for all positive integers $n,$ we get
an infinite chain
\begin{xlist}\item\label{x.a_0-oo}
$a_0\ <\ a_1\ <\ \dots\ <\ a_n\ <\ \dots\ .$
\end{xlist}
These form an infinite chain of
ideals of $P_1$ above $a_0,$
and the union of this chain will be a nonprincipal ideal strictly
larger than $a_0,$ contradicting the assumed maximality.\\[.5em]
{\it Proof 2.}
By the observations at the beginning of the proof
of Lemma~\ref{L.n-cmpct}, for each $n>0$ the posets of
$\!n{-}1\!$-compact elements and of $\!n\!$-compact elements of
$P_0\cong\id^n(P_n)$ are isomorphic respectively
to $\id(P_n)$ and to $P_n;$ comparing these statements for two
successive values of $n,$ we conclude that $\id(P_n)\cong P_{n-1}.$
This suggests that we extend the sequence of posets $P_n$
to allow negative subscripts by writing $P_{-n}=\id^n(P_0).$
Now let $Q$ be the disjoint union $\coprod_{n\in\mathbb{Z}}P_n,$
where elements from distinct posets $P_n$ are taken to be incomparable.
No ideal of $Q$ can contain elements of more than one of the $P_n,$
hence
\begin{xlist}\item[]
$\id(Q)\ =\ \coprod_{n\in\mathbb{Z}}\id(P_n)\ \cong
\ \coprod_{n\in\mathbb{Z}}P_{n-1}\ =
\ \coprod_{m\in\mathbb{Z}}P_m\ =\ Q.$
\end{xlist}

Hence by Ern\'e's result, $Q$ has ascending chain
condition; hence so does $P_0\subseteq\nolinebreak Q.$\qed\vspace{.5em}

It is interesting to compare the situation of the preceding theorem
with what we get if we start with any poset $P$ with a nonprincipal
ideal $I,$ and consider the posets
\begin{xlist}\item\label{x.P->id->}
$P\ \to\ \id(P)\ \to\ \id^2(P)\ \to\ \dots\ ,$
\end{xlist}
with connecting maps $d_{\id^{n-1}(P)}: \id^{n-1}(P)\to\id^n(P).$
Here $I$ can be regarded as an element $b_1\in\id(P),$ which is the
least upper bound therein of the set $d_P(I).$
On the other hand, the ideal of $\id(P)$ generated by
that set, since it consists of elements $<b_1,$ can be regarded as
an element $b_2\in\id^2(P)$ which is $<d_{\id(P)}(b_1);$
this element in turn will strictly majorize all elements
of $d_{\id(P)}d_P(I),$ and so the ideal generated by
that set in $\id^2(P)$ will be an element of $\id^3(P)$
which is $<d_{\id^2(P)}(b_2);$ and so on.
Letting $P_\infty$ denote the direct limit of~(\ref{x.P->id->}), and
writing $d_{\infty,n}: \id^n(P)\to P_\infty$ for the induced maps
to that object, we get a {\em descending} chain
$d_{\infty,1}(b_1)>d_{\infty,2}(b_2)>\dots$ above the
set $d_{\infty,0}(I)$ in $P_\infty.$
On the other hand, if we pause after $n$ steps, and consider the
chain $d_{\id^n(P)}\dots d_{\id(P)}(b_1)>
d_{\id^n(P)}\dots d_{\id^2(P)}(b_2)>\dots>b_n,$ this is essentially
the finite chain described in Lemma~\ref{L.n-cmpct}, used there
in building up the {\em ascending} chain~(\ref{x.a_0-oo}).

I don't know whether the analog of Theorem~\ref{T.id^n}
with $\cid$ in place of $\id$ is true.
(This seems related to the problem stated at the end of~\cite{MEr}.)
The natural approach to adapting the above argument to that case
would start by defining an element $x$ of a poset to be
{\em chain-compact} if
every chain $S$ having a least upper bound $\bigvee S$
which majorizes $x$ contains an element $s$ that already does so.
However, it turns out that elements of
$d_P(P)\subseteq\cid(P)$ are not necessarily chain-compact:
If, slightly modifying the example by which we showed in \S\ref{S.def}
that $\cId$ of a lattice need not be a lattice, we let
$P=(\omega\times\omega_1)\cup\{(\omega,\omega_1)\},$
i.e., that original example, with
the chains $\omega\times\{\omega_1\}$ and $\{\omega\}\times\omega_1$
deleted, but the top element $(\omega,\omega_1)$ retained, we find
that $\cid(P)$ can be identified with $(\omega+1)\times(\omega_1+1),$
in which the image of that top element is the least upper bound
of each of the chains $\omega\times\{\omega_1\}$ and
$\{\omega\}\times\omega_1,$ hence not chain-compact.
(Though it happens that the top element {\em was} chain-compact in $P.)$
This is related to the
fact that the inverse image under $d_P$ of the ideal generated
by either of these chains is a non-chain-generated ideal of $P.$
One encounters similar phenomena on taking for $P$ the poset of finite
subsets of a set of cardinality $\aleph_1,$ together with
the improper subset.

These examples used uncountable chains; might
the analog of Theorem~\ref{T.id^n} hold with $\id$ replaced
by the operator taking $P$ to its poset of nonempty ideals generated by
countable chains; equivalently, nonempty ideals with
countable cofinal subsets?
Example~3 of \cite{MEr} shows that this, too, fails: the poset $P$ of
that example, the totally ordered set $\omega_1,$
is easily seen to be isomorphic its own poset of countably
generated ideals, equivalently, bounded ideals (whether or not
we include the empty ideal).
The reason proof~1 of Theorem~\ref{T.id^n} fails to give a
contradiction in this case lies not in the phenomena
sketched above (indeed, the inverse image
in $\omega_1$ of a bounded ideal of $\cid(\omega_1)$
will again be a bounded ideal), but in the fact
that Zorn's lemma cannot produce a {\em maximal} bounded ideal.

\vspace{.5em}
A question suggested by juxtaposing the present considerations with
those of \cite[\S7]{P_vs_cP} is:  What can be said about lattices $L$
such that $\id(L)$ is (not necessarily equal to, but at least)
{\em finitely generated over} its sublattice $d_L(L);$
and similarly for upper semilattices?
(In these questions it makes no difference whether we
refer to $\id(L)$ or $\Id(L).)$

\vspace{.5em}
Since dropping the bottom element $\emptyset$ from $\Id(P)$ makes
such a difference in the properties we have studied,
it might be interesting to investigate
the effect on these questions of dropping the top element, $P,$
of $\id(P)$ or $\Id(P)$ if $P$ is a directed poset (e.g., a
lattice or semilattice); or of adding an
extra top element; though these constructions are admittedly
less natural than that of dropping $\emptyset.$
One might also investigate the variants of some of the questions
we have considered that one gets by using the opposite structures,
$\Id(P)^\r{op}$ etc., in place of $\Id(P)$ etc..

\vspace{.5em}
I will mention one other interesting result on the relation
between $L$ and $\id(L)$ for any lattice $L$: It is shown in
\cite{B+H} that $\id(L)$ is a homomorphic
image of a sublattice of an ultrapower of $L.$

\section{Appendix: On chains, and products of chains.}\label{S.chains}

We noted in \S\ref{S.def} that for $L$ a lattice, the poset
$\cId(L)$ of ideals of $L$ generated by chains need neither be an
upper nor a lower semilattice; our counterexample was based
on the fact that a direct product of two chains of
distinct infinite cofinalities has no cofinal subchain.
Let us put this phenomenon in a more general light.

\begin{lemma}\label{L.cof+prods}
Let $X$ be a class of posets, and for any poset $P,$
let $X\hy\r{Down}(P)\subseteq\r{Down}(P)$ denote the
set of all downsets $d\subseteq P$ of the form $d=P\downarrow f(Q),$
for $Q\in X$ and $f:Q\to P$ an isotone map.
Then\\[.5em]
\textup{(i)} The following conditions on $X$ are equivalent.

\vspace{.2em}
\textup{(i.a)} \ $X\hy\r{Down}(P)\,\subseteq\,\Id(P)$ for
all posets $P.$

\textup{(i.b)} \ Every member of $X$ is upward directed.%
\\[.5em]
\textup{(ii)} Among the following conditions on $X,$
we have the implication \textup{(ii.a)}$\implies$\textup{(ii.b)},
and, if the equivalent con\-ditions of \textup{(i)} above hold, also
\textup{(ii.a)$\implies$(ii.c)$\wedge$(ii.d)}.

\vspace{.2em}
\textup{(ii.a)} \ For all $Q,\,Q'\in X,$ there exists
$R\in X$ which admits an isotone map to
the product poset $Q\times Q',$ with cofinal image.

\textup{(ii.b)} \ $X\hy\r{Down}(S)$ is a lower
subsemilattice of $\r{Down}(S)$ for all lower semilattices $S.$

\textup{(ii.c\<)} \ $X\hy\r{Down}(S)$ is an upper subsemilattice
of $\Id(S)$ for all upper semilattices $S.$

\textup{(ii.d)} \ $X\hy\r{Down}(L)$ is a
sublattice of $\Id(L)$ for all lattices $L.$
\end{lemma}

\begin{proof}
Since an isotone image of an upward directed set is upward directed,
and the downset generated by an upward directed set is an ideal,
we clearly have (i.b)$\implies$(i.a).
Conversely, if~(i.b) fails, let $Q\in X$ not be upward directed.
Then $Q=(Q\downarrow Q)\in X\hy\r{Down}(Q)$ is
not an ideal, so~(i.a) fails.

To get~(ii), consider any $Q,\,Q'\in X,$ any poset $P,$
and any isotone maps $f:Q\to P,$ $f':Q'\to P.$
If our $P$ is a lower semilattice, then the intersection
of downsets $(P\downarrow f(Q))\cap(P\downarrow f'(Q'))$ can be
described as $P\downarrow\{f(q)\wedge f'(q')\mid q\in Q,\ q'\in Q'\},$
while if $P$ is an upper semilattice and $P\downarrow f(Q)$ and
$P\downarrow f'(Q')$ are ideals, then their join in $\Id(P)$ can
be described as $P\downarrow\{f(q)\vee f'(q')\mid q\in Q,\ q'\in Q'\}.$
In these statements, note that the sets
$\{f(q)\wedge f'(q')\mid q\in Q,\ q'\in Q'\},$
respectively $\{f(q)\vee f'(q')\mid q\in Q,\ q'\in Q'\},$
are isotone images of the poset $Q\times Q'.$
Hence if $X$ contains a poset $R$ which admits an isotone map
$g:R\to Q\times Q'$ with cofinal image, the composite of $g$ with
the above map $Q\times Q'\to P$ is a map $R\to P$ whose image
generates the indicated meet-downset or join-ideal respectively.
This gives (ii.a)$\implies$(ii.b), and, assuming the conditions
of~(i), also (ii.a)$\implies$(ii.c);
together these give (ii.a)$\implies$(ii.d).
\end{proof}

To avoid awkward statements, I have not attempted to formulate
if-and-only-if versions of the implications of~(ii).
That the converses to the present statements do not hold arises
from the fact that on members of $X$ we are only assuming a poset
structure, but we are mapping them
into sets with lattice or semilattice structure.
For instance, since in a lower semilattice $S$ every downset is a
connected poset, we see that the class $C$ of all connected posets
satisfies $C\,\hy\r{Down}(S)=\r{Down}(S);$
hence taking $X=C\<\cup\{Q\},$ where $Q$ is the disconnected
poset consisting of two incomparable elements, we find that
$X\hy\r{Down}(S)$ is still $\r{Down}(S),$ so~(ii.b) holds.
But~(ii.a) does not, since no member of $X$ can be mapped into
$Q\times Q$ so as to have cofinal image.

However, the above lemma shows why the choice for $X$ of
the class of all chains (or even the set consisting
of the two chains $\omega$ and $\omega_1)$ can
fail to have properties (ii.c) and~(ii.d), and points to
some variants that will have those properties.
Any class of upward directed posets closed under
taking pairwise products will satisfy (i.b) and~(ii.a), and
hence (ii.b)-(ii.d); in particular, this will be true of the class
of all finite products of chains (cf.\ \cite{MP}).
A singleton whose one member is a chain, $Q,$
will also satisfy these properties, since the diagonal image of $Q$
in $Q\times Q$ is cofinal.
Both of these classes yield variants of the construction $\cId$ that
are, in this respect, better behaved than that construction.

% - - - - - - - - - - - - - - - - - - - - - - - - - - - - - -

\vspace{.5em}
\noindent
Department of Mathematics\\
University of California\\
Berkeley, CA 94720-3840\\
USA\\
gbergman@math.berkeley.edu

\begin{thebibliography}{00}

\bibitem{B+H} Kirby A. Baker and Alfred W. Hales,
{\em From a lattice to its ideal lattice,}
Algebra Universalis {\bf 4} (1974) 250--258. % QA251.A341
~MR~{\bf 51}\#291.

\bibitem{P_vs_cP} George M. Bergman,
{\em Two statements about infinite products that are not quite true,} in
{\em Groups, Rings \& Algebras} (Proceedings of a conference in honor of
Donald S. Passman), ed. W. Chin, J. Osterburg, and D. Quinn.
Contemporary Mathematics v.{\bf 420} (2006) 35--58.
~MR~{\bf 2007k}:16008.

\bibitem{embed} George M. Bergman,
{\em Some results on embeddings of algebras, after de Bruijn
and McKenzie,} Indag. Math.,
{\bf 18} (2007) 349--403.
arXiv:0704.0275.
~MR~2373687.
Update at
http://math.berkeley.edu/{$\!\sim$}gbergman/papers/updates/embed.html\,.

\bibitem{GBk} Garrett Birkhoff,
{\em Lattice theory,} 3rd ed..
AMS Colloquium Publications, Vol. XXV, 1967.
~MR~{\bf 37}\#2638.

\bibitem{MEr1} Marcel Ern\'e,
{\em Adjunctions and standard constructions for partially ordered sets,}
pp.\,77--106 in {\em Contributions to general algebra, 2}
(Klagenfurt, 1982), 1983.
~MR~{\bf 85d}:06002.

\bibitem{MEr} Marcel Ern\'e,
{\em Posets isomorphic to their extensions,}
Order {\bf 2} (1985) 199--210.
~MR~{\bf 87b}:06006.

\bibitem{MEr2} Marcel Ern\'e,
{\em Compact generation in partially ordered sets,}
J. Austral. Math. Soc., Ser. A {\bf 42} (1987) 69--83.
~MR~{\bf 88a}:06002.

% \bibitem{MEr3} Marcel Ern\'e,
% {\em Minimal bases, ideal extensions, and basic dualities,}
% Proc. 19th Summer Conf. on Topology and its Applications.
% Topology Proc. {\bf 29} (2005) 445--489.
% ~MR~{\bf 2007g}:06004.

\bibitem{OF} Orrin Frink,
{\em Ideals in partially ordered sets,}
Amer. Math. Monthly {\bf 61} (1954) 223--234.
~MR~{\bf 15},848a.

\bibitem{G+D} A. M. Gleason and R. P. Dilworth,
{\em A generalized Cantor theorem,}
Proc. Amer. Math. Soc. {\bf 13} (1962) 704--705.
~MR~{\bf 26}\#2365.

\bibitem{DH} Denis Higgs,
{\em Lattices isomorphic to their ideal lattices,}
Algebra Universalis {\bf 1} (1971) 71--72. % QA251.A341
~MR~{\bf 45}\#123.

\bibitem{DK} Georges Kurepa,
{\em Ensembles ordonn\'es et leurs sous-ensembles bien ordonn\'es},
C. R. Acad. Sci. Paris {\bf 242} (1956) 2202--2203.
~MR~{\bf 17},1065j (but the assertion in this review that the
first result of the paper is incorrect is itself incorrect).

\bibitem{HM} H. M. MacNeille,
{\em Partially ordered sets},
Trans. Amer. Math. Soc. {\bf 42} (1937) 416--460.
~MR~1501929. % pre MR paper, so no review

\bibitem{M+M+T} Ralph N. McKenzie, George F. McNulty and
Walter F. Taylor,
{\em Algebras, lattices, varieties. Vol. I,}
Wadsworth \& Brooks/Cole, 1987.
~MR~{\bf 88e}:08001.

\bibitem{MP} M. Pouzet,
{\em Parties cofinales des ordres partiels ne contenant pas
d'anticha\^{i}nes infinies}, 1980, preprint.

\bibitem{ST} Stevo Todorcevic,
{\em Partition relations for partially ordered sets},
Acta Math. {\bf 155} (1985) 1--25.
~MR~{\bf 87d}:03126.

\bibitem{FW} Friedrich Wehrung,
{\em Embedding coproducts of partition lattices,}
Acta Sci.\ Math.\ (Szeged) {\bf 73} (2007) 429--443.
~MR~2380058.

\end{thebibliography}
\end{document}